\documentclass[12pt]{amsart}
\usepackage{amssymb}

\newtheorem{thm}{Theorem}
\newtheorem{lemma}[thm]{Lemma}
\newtheorem{prop}[thm]{Proposition}

\theoremstyle{definition}
\newtheorem{remark}[thm]{Remark}

\def\al{\alpha}
\def\ep{\varepsilon}

\def\ph{\varphi}

\def\ep{\varepsilon}
\def\vf{\varphi}
\def\NN{\mathbb N}
\def\RR{\mathbb R}
\def\R{\mathbb R}

\def\diam{\operatorname{diam}}

\def\dist{\operatorname{dist}}
\def\conv{{\rm conv\,}}
\def\cl{\overline}

\begin{document}

\title [Quotients of convex functions on nonreflexive Banach spaces]
{Quotients of  continuous convex functions on nonreflexive Banach spaces}

\author{P. Holick\'y, O. Kalenda, L. Vesel\'{y}, and L. Zaj\'\i\v{c}ek}

\address{Faculty of Mathematics and Physics\\ Charles University\\ Sokolovsk\'a~83\\ 186~75, Praha~8\\ Czech Republic
(P.H., O.K. and L.Z.)\newline
Dipartimento di Matematica ``F.~Enriques'', Universit\`a degli Studi di Milano,
Via C.~Saldini 50, 20133 Milano, Italy (L.V.)}

\email{holicky@karlin.mff.cuni.cz, kalenda@karlin.mff.cuni.cz, \newline
       vesely@mat.unimi.it, zajicek@karlin.mff.cuni.cz}

\thanks{ The third author was supported 
in part by the Ministero dell'Universit\`a e della Ricerca of Italy. The other authors were supported by MSM 0021620839 financed by MSMT of Czech Republic, by
        GA\v CR 201/06/0198 and GA\v CR 201/06/0018.}

\begin{abstract}
On each nonreflexive Banach space $X$ there exists a positive continuous convex function 
$f$ such that $1/f$ is not
 a d.c.\ function (i.e., a difference of two continuous convex functions). 
This result together with known ones implies that $X$ is reflexive if and only if each 
everywhere defined quotient of two continuous convex functions is a d.c. function. 
Our construction gives also a stronger version of Klee's result concerning renormings 
of nonreflexive spaces and non-norm-attaining functionals.
\end{abstract}



\maketitle

\bigskip

\smallskip

A function on a Banach space $X$ is called a {\it
d.c.\ function} if it can be represented as a difference of two
continuous convex functions (all functions considered in this note are
real-valued). Thus the system of all d.c.\ functions on $X$
is the smallest vector space containing all continuous convex functions.
Moreover, it is well-known, and not difficult to show, that it is even
an algebra and a lattice (see, e.g., \cite[III.2]{HU}). While an
everywhere defined quotient $g/f$ of two
d.c.\ functions on a finite-dimensional Banach space
is always d.c. (cf.\ \cite[Corollary]{Ha}), the situation  is completely different
for infinite-dimensional spaces:
by \cite[Corollary 5.7]{VZ}, on each infinite-dimensional Banach space there exists a
positive d.c.\  function such that $1/f$ is not
d.c. 

The following natural {question} arises:

\begin{enumerate}
\item[]
{\it
is the
quotient $g/f$ of two continuous \underbar{convex} functions on $X$ d.c.\
if $f\ne0\,$?}
\end{enumerate}

\noindent
Quite surprizingly, the answer is affirmative for all reflexive spaces
$X$; indeed, it is proved in
\cite[Remark 3.5(i)]{VZ} that $1/f$
($f\ne0$ continuous and convex)
is d.c.\ on $X$ whenever $X$ is reflexive.
The main aim of this note is to show that the above question has a negative
answer for each
nonreflexive Banach space $X$.

The following criterion for non-d.c.\ functions (cf.\ \cite[Lemma~4.1]{VZ})
suggests how to construct a counterexample.

\begin{lemma}\label{notdc}
Let $X$ be a Banach space and $h\colon X \to \R$ be a function. If
there exist sets $M\subset X$ of arbitrarily small diameter such that
$h$ is unbounded on $M$,
then $h$ is not a d.c.\ function.
\end{lemma}

\noindent If there exists a continuous convex function $f$ on $X$
such that

\vskip-0.4truecm

\begin{equation}\label{funkce}
\text{$f>0$, and
there exist sets $M$ of arbitrarily small
diameters with
$\inf f(M)=0\,$,}
\end{equation}

\noindent
then $1/f$ is not a d.c.\ function by Lemma~\ref{notdc}.
(Of course, such an $f$ cannot exist if $X$ is reflexive since, in this case,
 $f$ attains its minimum on any closed ball.)

To construct $f$, it might seem natural to proceed by finding an $x^*\in X^*$ such that

\vskip-0.4truecm

\begin{equation}\label{xsep}
\begin{matrix}
\text{$x^*$ does not attain its norm, and  }\\
\text{there exist sets $M\subset B_X$ of arbitrarily
small diameter such that $\sup x^*(M)=\|x^*\|_*$.}
\end{matrix}
\end{equation}

\noindent
Indeed, if we had such an $x^*$, it would be sufficient to put $f(x):= \|x\|-\|x^*\|_*$ if
$x^*(x)=\|x^*\|_*$, and to extend $f$ to the whole $X$ so that $f$ is
constant on each line parallel to a fixed vector $v\in X$
such that $x^*(v)\neq 0$. While it is not difficult to check that
no such $x^*$ exists in the
classical nonreflexive spaces $c_0$ and $\ell_1$ (with their canonical norms), it is possible
to prove (see below) that such an $x^*$ always exists after a suitable
equivalent renorming of (a nonreflexive) $X$.

However, we proceed in a different order.
First, using a James' sequential
characterization of nonreflexivity,
we construct a continuous convex function $f$ on $X$, satisfying
(\ref{funkce}),
as a distance function from a certain bounded convex set in $X\oplus\RR$.
Using this $f$, we easily prove our main Theorem 4, which also gives a modification 
of the well known characterization
 of nonreflexive spaces by monotone sequences of closed convex sets. 
Then, using the existence of such $f$ on each hyperplane of $X$, we show  that, 
if $X$ is nonreflexive,
each nonzero functional $x^*\in X^*$ satisfies (\ref{xsep}) with respect
to a suitable equivalent norm on $X$. This last assertion is the content of
Proposition~\ref{renorming} which we believe to be of independent interest
since it improves the following
Klee's result \cite{Kl}: each nonzero bounded
linear functional on a nonreflexive Banach space $X$
is not norm-attaining for some equivalent norm on $X$.

\medskip

\noindent

Let us start by fixing some notations. We consider only Banach spaces over the reals $\R$.
We denote by $B_X$ or $B_{(X,\|\cdot\|)}$ the closed unit
ball in a Banach space $X$ endowed with a norm $\|\cdot\|$.
By $\|\cdot\|_*$ we denote the corresponding dual norm on $X^*$ (the topological dual of $X$).

In what follows, we consider $X \oplus \R$  equipped with the maximum norm,  and  we
 identify $x \in X$ with $(x,0) \in  X \oplus \R$ (and so $X$ with $X \times \{0\}$).

\begin{lemma}\label{convexfunction}
Let $X$ be a nonreflexive Banach space.
Then there exists a nonempty bounded convex set $C \subset X \oplus \R$ such that
\begin{itemize}
\item[(a)] $\ph(x):= \dist(x,C)>0$ for every $x\in X$, and
\item[(b)] for each $\ep>0$ there is a set $M_{\ep} \subset X$ with $\diam M_{\ep} < \ep$ and
$\inf\ph(M_{\ep})=0$.
\end{itemize}
\end{lemma}

\begin{proof}
Since $X$ is nonreflexive, by \cite[Theorem 1]{J} (see, e.g., \cite[Theorem 10.3]{Va} or \cite[Theorem 1.13.4]{M}
 for simpler proofs) there exist unit vectors $\{e_i\}_{i=1}^{\infty}$ in $X$ and unit functionals
$\{e^*_i\}_{i=1}^{\infty}$ in $X^*$
such that
\begin{equation}\label{bipo}
e^*_i(e_j)=0\ \ \text{if}\ \ i>j,\ \ \text{ and}\  \
 e^*_i(e_j)> 1/2\ \ \text{ if}\ \ i\le j.
\end{equation}
Set $e_{\infty}:=(0,1)\in X\oplus\RR$, and let
$f_i\in (X\oplus\RR)^*$ be the extension of $e_i^*$ for which $f_i(e_{\infty})=1$.
Clearly $\|f_i\|_*=2$.
For $0<k<n$ in $\NN$, we define
$$x_{k,n}:=2e_k+\frac{2}{k}e_n+\frac{1}{n}e_{\infty}.$$
Clearly
\begin{equation}\label{male}
f_i(x_{k,n})\geq 1\ \ \text{ for}\ \ 1\leq i \leq k,
\end{equation}
\begin{equation}\label{stre}
f_i(x_{k,n})\geq \frac{1}{k}\ \  \text{for}\ \ k < i \leq n,\ \ \text{ and}
\end{equation}
\begin{equation}\label{velk}
f_i(x_{k,n})=\frac{1}{n}\ \  \text{ for}\ \ i > n.
\end{equation}
We define $C:=\conv\{x_{k,n}:0<k<n,\ k,n\in\NN\}$ and $X_0:=\cl{\mathrm{span}}\{e_j:j\in\NN\}$.

To prove (a), we need to show $\cl{C}\cap X =\emptyset$. Since clearly $\cl{C}\cap X \subset X_0$, it is  sufficient to show that $\cl{C}\cap X_0 =\emptyset$.
 So, suppose to the contrary that an $x_0\in\cl{C}\cap X_0$ is given. As $\|f_i\|_*=2$ and $\lim_{i\to\infty}f_i(e_j)=0$ for each $j \in \NN$, it is easy to check that $\lim_{i\to\infty}f_i(x)=0$ for every $x\in X_0$.
So, we may find natural numbers $i_1 < i_2 < i_3$ such that
\begin{equation}\label{x0}
 f_{i_1}(x_0)<\frac{1}{3},\ \
 i_1\,  f_{i_2}(x_0)<\frac{1}{3},\ \
   i_2\,  f_{i_3}(x_0)< \frac{1}{3}.
\end{equation}
Since $x_0\in\cl{C}$ and $f_{i_1}$, $f_{i_2}$, $f_{i_3}$ are continuous, we can find $c\in C$ so close to $x_0$ that
\begin{equation}\label{c}
  f_{i_1}(c)<\frac{1}{3},\ \
 i_1\,  f_{i_2}(c)<\frac{1}{3},\ \
  i_2\,  f_{i_3}(c)< \frac{1}{3}.
\end{equation}
Since  $c\in C$, we can assign to each $(k,n)$ with $1\leq k < n$ a number $\al_{k,n} \geq 0$ so that
 $\sum \al_{k,n} = 1$, the set $\{(k,n):\ \al_{k,n}\neq 0\}$ is finite, and $c = \sum \al_{k,n} \, x_{k,n}$.

 Using subsequently \eqref{male}, \eqref{stre}, and \eqref{velk}, we obtain

 \begin{equation}\label{jedna}
 f_{i_1}(c) = \sum \al_{k,n}f_{i_1}(x_{k,n}) \geq \sum_{\substack{k \geq i_1\\ n>k}} \, \al_{k,n}\,,
 \end{equation}

 \begin{equation}\label{dve}
 f_{i_2}(c) = \sum \al_{k,n}f_{i_2}(x_{k,n}) \geq 
\sum_{\substack{k <i_1\\ n\geq i_2}} \, \frac{1}{k}\,\al_{k,n}\, 
\ge\frac{1}{i_1}\,\sum_{\substack{k <i_1\\ n\geq i_2}} \, \al_{k,n} \,,
 \end{equation}

 \begin{equation}\label{tri}
 f_{i_3}(c) = \sum \al_{k,n}f_{i_3}(x_{k,n}) \geq 
\sum_{\substack{k <i_1\\ n < i_2}} \, \frac{1}{n}\,\al_{k,n} 
\ge \frac{1}{i_2}\,\sum_{\substack{k <i_1\\ n < i_2}} \, \al_{k,n} \,.
 \end{equation}
 Using \eqref{jedna}, \eqref{dve}, \eqref{tri} and \eqref{c}, we easily obtain $\sum \al_{k,n} <1$, which is a contradiction.

To prove (b), consider an arbitrary $\ep>0$. Choose $k_0 \in \NN$ with $4/k_0 < \ep$ and set
 $M_{\ep}: = \{ 2e_{k_0} + (2/k_0) e_n: n> k_0\}$. Then clearly $\diam M_{\ep} \leq 4/k_0 <\ep$. The other property
  of $M_{\ep}$ also holds, since, for each $n> k_0$,
  $$ \inf\vf(M_{\ep}) = \dist(M_{\ep},C) \leq \|( 2e_{k_0} + (2/k_0) e_n)
-  (2e_{k_0} + (2/k_0) e_n + (1/n)e_{\infty})\| = 1/n.$$
\end{proof}

\begin{remark}\label{simpler}
\begin{enumerate}
\item[(i)]
To obtain $C$ with the weaker property  $\inf_{x \in X} \vf(x) =0$ instead of (b) in
Lemma~\ref{convexfunction},
 it is sufficient to put $C:= \conv \{2e_k + (1/k) e_{\infty}:\ k \in \NN\}$,
 and the proof  becomes simpler.
 \item[(ii)]
 Setting $C:=\conv\{2e_k+\frac{2}{k}e_n+ \frac{2}{n}e_m + \frac{1}{m}e_{\infty} :0<k<n<m,\ k,n,m\in\NN\}$,
an easy modification of the proof of Lemma~\ref{convexfunction} gives
the following property $\mathrm{(b^2)}$ which is slightly stronger than (b):
\begin{enumerate}
\item[$\mathrm{(b^2)}$]
{\em there exist sets $M\subset X$ of arbitrarily small diameter such that
$M$ contains sets $A$ of arbitrarily small diameter with
$\inf\ph(A)=0$.}
\end{enumerate}
\item[]  (Analogously,
using indices $0<k_1< \dots  <k_{p+1}$ in the definition of $C$, it is
possible to obtain the corresponding iterated property
($\mathrm{b}^p$).)
\end{enumerate}
 \end{remark}

\medskip


Now, we are ready to state the following main result of the present
paper.

\begin{thm}\label{eqrefl}
The following properties of a Banach space $X$ are equivalent.
\begin{itemize}
    \item [(a)] $X$ is nonreflexive.
    \item [(b)] There is a continuous convex function $f\colon X\to(0,\infty)$ such that $1/f$
                    is not representable as a difference of two continuous convex functions.
    \item [(c)] There is a decreasing sequence $\{C_n\}_{n=1}^\infty$ of bounded closed
                    convex subsets of $X$ such that
    \begin{equation*}
\bigcap_{n=1}^{\infty} C_n=\emptyset\,,\ \ \text{and}\ \
    \bigcap_{n=1}^{\infty} (C_n + \ep B_X) \neq \emptyset\text{ for every }\ep>0.
     \end{equation*}
\end{itemize}
\end{thm}

\begin{proof}
If $X$ is nonreflexive, take $f:=\ph$ where $\ph$ is as in
Lemma~\ref{convexfunction}. By Lemma~\ref{notdc}, $1/f$ is not d.c.\ on
$X$. On the other hand, if $X$ is reflexive and $f$ is a positive
continuous convex function, then $1/f$ is d.c.\ on $X$ by
\cite[Remark 3.5(i)]{VZ}. Thus (a) and (b) are equivalent.

Let us show that (a) and (c) are equivalent. If $X$ is nonreflexive,
let $\ph$ be again the function from Lemma~\ref{convexfunction}. The sets
$C_n:=\{x\in X: \ph(x)\le 1/n\}$, $n\in\mathbb{N}$, are nonempty, closed, convex, bounded
(since the set $C$ in Lemma~\ref{convexfunction} is bounded) and their
intersection is empty. Let $\ep>0$. By the properties of $\ph$, there exists $x\in X$ such
that, for each $n$, there is $y\in B(x,\ep)$ with $\ph(y)\le 1/n$, i.e. $y\in C_n$.
In other words, $x\in\bigcap_{n=1}^\infty (C_n+\ep B_X)$. Hence (a)
implies (c). On the other hand, if $X$ is reflexive, then each
decreasing sequence $\{C_n\}$ of nonempty closed bounded convex subsets
of $X$ has a nonempty intersection since each $C_n$ is weakly compact.
\end{proof}


\smallskip

Let us conclude our paper with the promissed strengthening of a result
from \cite{Kl}.

\begin{prop}\label{renorming}
Let $Y$ be a nonreflexive Banach space and $0 \neq y^* \in Y^*$. 
Then there exists an equivalent norm $|\cdot|$ on $Y$ such that
 \begin{enumerate}
 \item[(a)]
 $y^*$ does not attain its norm on  $B_{(Y,|\cdot|)}$, and
 \item[(b)]
 for each $\ep>0$, there is $M_{\ep} \subset B_{(Y,|\cdot|)}$ such that
$\diam M_\ep<\ep$ and 
$\sup y^*(M_\ep) = |y^*|_*$.
 \end{enumerate}
\end{prop}

\begin{proof}
 Set $X := \{y \in Y:\ y^*(y) = 0\}$ and choose $e\in Y$ with $y^*(e)=1$. Up to renorming, 
we may suppose that the norm on $Y$ satisfies 
$$\|y\|=\max\{\|y-y^*(y)e\|,|y^*(y)|\}$$
for all $y\in Y$. In this way we may identify
$Y$ with $X \oplus_\infty \R$ so that
$y^*((x,t))= t$ for $(x,t) \in X \times \R$. 
   
As $Y$ is not reflexive, $X$ is not reflexive, either.
Let $\varphi$ be the function on $X$ given by Lemma~\ref{convexfunction}.
Choose $\alpha>\varphi(0)$ and set
$$A=\{x\in X:\varphi(x)<\alpha \}.$$
By the properties of $\varphi$ the set $A$ is bounded. Therefore we can choose $r>0$ such that $A\subset B(0,r)$. 
Choose  $\beta >\sup\varphi(B(0,r))$; it is possible as $\varphi$ is $1$-Lipschitz. 
Further define  
\begin{align*}
D&=\{(x,t)\in X\times\RR: x\in B(0,r),\; t=\varphi(x)-\beta \},\\
C&=\overline{\mathrm{conv}}\, (D \cup (-D)).
\end{align*}
Then $C$ is clearly a bounded closed convex symmetric set.  Further, 
$0\in\operatorname{int}C$, as  $0\in A$ and $A\times(\alpha-\beta,\beta-\alpha)\subset C$.
It follows that there exists an equivalent norm $|\cdot|$ on $X\times \RR$ such that $C$ is the closed unit ball in this norm. We will show that this norm has the required properties.
      
We have
%
\[
-|y^*|_*
=\inf y^*(C) = \inf y^*\bigl(D\cup(-D)\bigr) 
= \inf y^*(D)
=\inf\{\varphi(x)-\beta :x\in B(0,r)\}
=-\beta,
\]
as clearly $\inf\ph(B(0,r))=\inf\ph(X)=0$.
%
%
Thus $|y^*|_*=\beta$.
 
Next we are going to show that $y^*$ does not attain its norm on $C$. Suppose it does.
Then there is a point $z=(x_0,-\beta)\in C$.
Note that
$$C\subset \{ (x,t)\in X\times\RR : x\in B(0,r)\ \&\ t\ge \varphi(x)-\beta \}.$$
The reason is that the set on the righthand side is closed and convex and it contains both $D$ and $-D$. It follows that $z$ belongs to the set on the righthand side, i.e. 
$-\beta\ge\varphi(x_0)-\beta$. So $\varphi(x_0)\le 0$, a contradiction.

It remains to show the assertion (b). Let $\ep >0$ be given. By the properties of $\varphi$ 
we can choose a set $P_{\ep} \subset A$ such that $\diam P_{\ep} < \ep$ and $\inf\varphi({P_{\ep}}) = 0$. (Note that $\varphi\ge\alpha$ outside of $A$.) Now set 
       $$ P^*_{\ep} := \{(x,t)\in X\times\RR:\ x \in P_{\ep},\; t = \varphi(x) - \beta\}.$$
 Then clearly  $ P^*_{\ep} \subset C$ and 
 $$\inf_{z \in  P^*_{\ep}} y^*(z) = -\beta=-|y^*|_*.$$
As $\varphi$ is $1$-Lipschitz with respect to $\|\cdot\|$, we get that $\|\cdot\|$--$\diam P^*_\ep<\ep$. Set $M_{\ep} := -  P^*_{\ep/K}$, where $K>0$ is such that $|\cdot|\le K\|\cdot\|$ on $X\times \RR$. Then $M_{\ep}$ has all required properties and the proof is complete.
\end{proof}

\end{document}